\documentclass[12pt,reqno]{amsart}
\usepackage{amsfonts,amsmath,amscd,amssymb,amsthm,color,graphicx}
\usepackage{enumerate,hyperref,perpage,url}
\usepackage[margin=2.2cm, a4paper]{geometry}

\theoremstyle{plain}
\newtheorem{thm}{Theorem}

\newtheorem{prop}[thm]{Proposition}

\theoremstyle{remark}
\newtheorem{defn}[thm]{Definition}

\newtheorem{example}{\textbf{Example}}
\newtheorem{nexample}{\textbf{Non-Example}}

\numberwithin{equation}{section}
\numberwithin{thm}{section}

\MakePerPage{footnote} \allowdisplaybreaks 

\newcommand\N{\mathbb N}

\newcommand\R{\mathbb R}

\newcommand{\Hess}{\text{Hess}}

\title[Stationary phase type estimates for low symbol regularity]{Stationary phase type estimates for low symbol regularity}

\author{Melissa Tacy}
\address{Department of Mathematics and Statistics, University of Otago, Dunedin Otago, New Zealand}
\email{mtacy@maths.otago.ac.nz}

\date{}

\begin{document}
\maketitle

\begin{abstract}
The well-known stationary phase formula gives us a way to find asymptotics of oscillating integrals so long as the symbol is regular enough (in comparison to the large parameter controlling the oscillation). However in a number of applications we find ourselves with symbols that are not suitably regular. In this paper we obtain decay bounds for such oscillatory integrals.
\end{abstract}

The stationary phase formula gives us a way to compute integrals of the form
$$I(\lambda)=\int_{\R^{n}} e^{i\lambda\phi(x)}a(x)dx\quad\lambda\in\R^{+}$$
under the condition that $\phi$ has a single stationary point, denoted $x_{c}$, on the support of $a(x)$ and
$$|\det(\Hess_{\phi})|\geq{}c>0.$$
In particular it gives us the asymptotic sum
\begin{equation}I(\lambda)=e^{i\lambda\phi(x_{c})}\sum_{j=0}^{\infty}\lambda^{-\frac{n-1}{2}-j}b_{j}(x_{c})\label{statphasexp}\end{equation}
where $b_{0}(x)$ is a constant multiple of the symbol $a(x)$, that is $b_{0}(x)=c_{n}a(x)$. The classical theory of these expansions is very well developed. See for example the texts of H\"{o}rmander \cite{H}, Muscalu-Schlag \cite{MS} or Stein \cite{S}. In many applications it is, however, necessary to deal with oscillatory integrals where  the phase functions and symbol have more complicated behaviour.  In \cite{ABZ} Alazard, Burq and Zuily study oscillatory integrals where the phase function and symbol have dependence on external parameters. These type of integrals frequently show up in the analysis of Fourier integral operators associated with the wave and Schr\"{o}dinger equations, (see for example \cite{ABZ2},\cite{BGT},\cite{GHT},\cite{koch},\cite{tacy09}). Many of the phase functions in these applications have form similar to
$$\phi(t,x,y)=t|x|^{2}-\langle y,x\rangle$$
so the Hessian of $\phi$ depends on $t$. For such integrals Alazard, Burq and Zuily find the precise decay estimates in terms of the external parameters. However they do not address the issue of symbol regularity or the case where the phase function may have some dependence on $\lambda$ itself. Cases where the phase function depend on $\lambda$ can arise when considering propagators that have some $\lambda$ dependence themselves and the necessitiy of placing cut off functions to localise results often cause symbols to lose nice regularity problems. In the papers \cite{BGT},\cite{koch}, \cite{tacy09} the authors treat such problems for some specific phase functions and symbols.  In this paper we consider a very general setting
 where the phase function $\phi$ and symbol $a$ are allowed to depend on $\lambda$ itself. In particular we wish to consider cases where there is a regularity loss of a positive power of $\lambda$ for each derivative. The conditions that we place on $a$ and $\phi$ are designed to encompass previous work but also with an eye to the current developments in the field and the technical needs for future work. As the conditions and related notation are rather technical we also (for the readers convenience) record our specific notation in the Appendix at the end of the paper.  For completeness (and because it often arises in applications) we will allow $I(\lambda)$ to  depend on other variables $y$ (which we assume to be in $\R^{d}$, $d$ need not be equal to $n$) in the following fashion
$$I(\lambda,y)=\int e^{i\lambda\phi(\lambda,x,y)}a(\lambda,x,y)dx.$$
We are particularly interested in the cases where for each $(\lambda,y)$ there is a single critical point satisfying 
$$\nabla_{x}\phi(\lambda,x,y)=0.$$
We will commonly denote this point $x(y)$, supressing the dependence on $\lambda$ which is generally mild. We will assume a lower bound on the determinant of the Hessian of $\phi$. That is for each $(\lambda,y)\in\R^{+}\times \R^{n}$ we have a lower bound
$$|\det(\Hess_{\phi})|\geq{}c\mu(\lambda,y)^{n}>0.$$
Cases of particular interest are those where $\mu(\lambda,y)\to 0$ as $\lambda\to\infty$ (the rate may depend on $y$). Throughout this paper we will abuse notation and write $\mu=\mu(\lambda,y)$. The $y$ dependence means that $I(\lambda,y)$ is now a function and we would like to know about its oscillation/regularity properties. If $a(\lambda,x,y)$ is smooth with derivative bounds independent of $\lambda$, that is for any multi-index $\gamma=(\gamma_{1},\dots,\gamma_{n})$ and
$$D_{x}^{\gamma}=\left(\frac{1}{i}\frac{\partial}{\partial x_{1}}\right)^{\gamma_{1}}\cdots \left(\frac{1}{i}\frac{\partial}{\partial x_{n}}\right)^{\gamma_{n}},\quad |\gamma|=\gamma_{1}+\cdots+\gamma_{n}$$
the symbol $a(\lambda,x,y)$ obeys the derivative bounds,
$$|D_{x}^{\gamma}a(\lambda,x,y)|\leq C_{\gamma},$$
then \eqref{statphasexp} extends to the case of $I(\lambda,y)$. Indeed studying the proof of \eqref{statphasexp}  carefully it is not hard to see that (for fixed $\mu=1$) if
$$|D_{x}^{\gamma}a(\lambda,x,y)|\leq{}C_{\gamma}\lambda^{|\gamma|\left(\frac{1}{2}-\epsilon\right)}$$
for some $\epsilon>0$ we can still make sense of the asymptotic sum as $\lambda\to\infty$. It is not difficult to extend \eqref{statphasexp} to these cases. In this note, we are however concerned with the case where the regularity of $a(\lambda,x,y)$ is far worse. For instance where each derivative of $a(\lambda,x,y)$ costs a factor of $\lambda^{\beta}$ for some $\frac{1}{2}\leq\beta\leq 1$,
$$|D_{x}^{\gamma}a(\lambda,x,y)|\approx \lambda^{|\gamma|\beta}\quad\beta\geq 1/2.$$ 
Now the asymptotic sum does not even necessarily converge. However if we study the question from the viewpoint that extreme oscillations of $e^{i\lambda\phi(x,y)}$ cause cancellation we can see that there is good reason to still expect some decay for $I(\lambda,y)$. Focusing only on a one dimensional example (with $\mu=1$) consider
$$I(\lambda)=\int_{\R^{n}} e^{i\lambda x^{2}}a(\lambda,x)dx.$$
At a heuristic level the stationary phase formula tells us that if $a$ is smooth independent of $\lambda$ the main contribution to $I(\lambda)$ comes from near $x=0$. After that the oscillations become rapid enough to force cancellation, see Figure \ref{goodreg}. 
\begin{figure}[h!]
\includegraphics[scale=0.5]{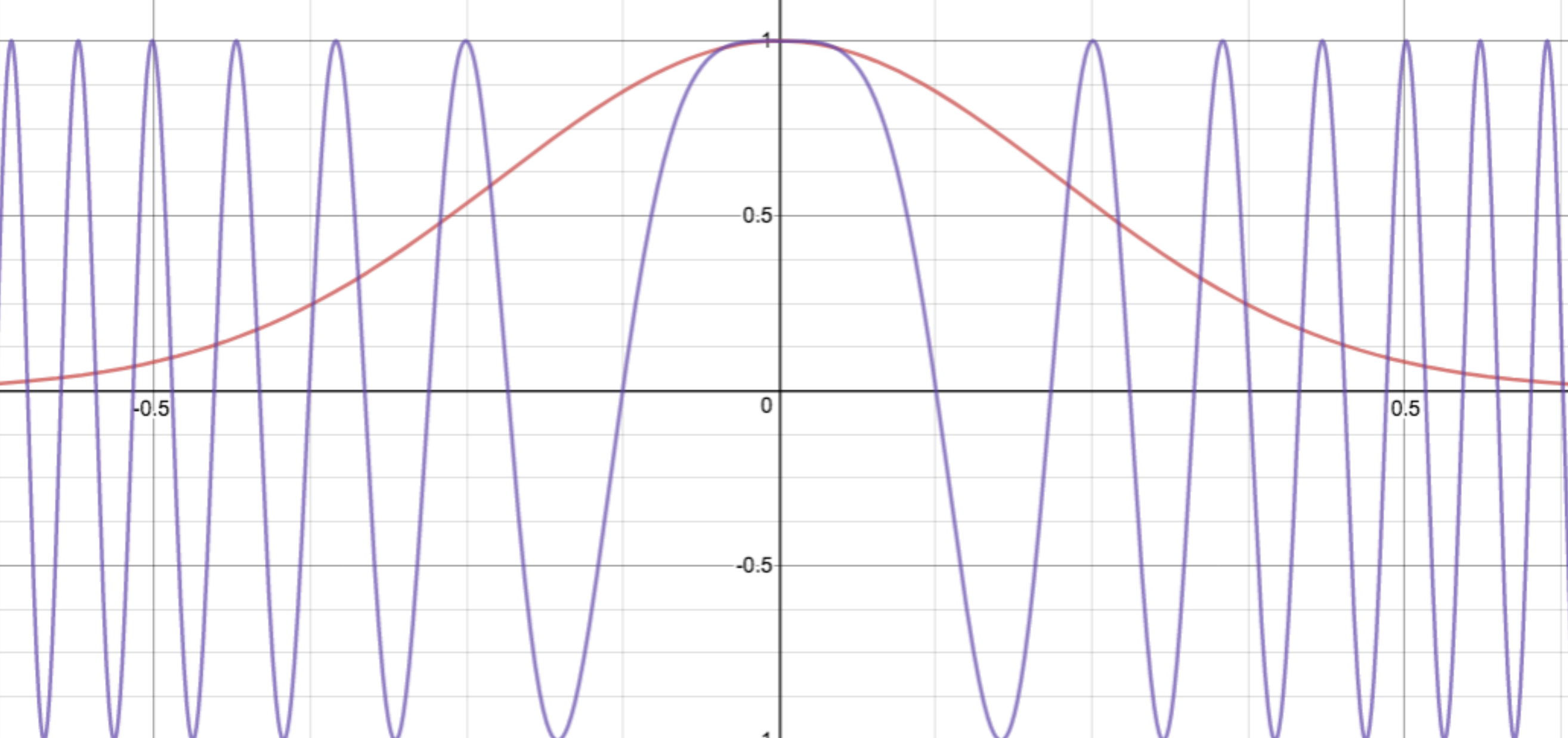}
\caption{The $y$ axes depicts both the real part of $e^{i\lambda x^{2}}$ and a symbol $a(\lambda,x)$ with regularity loss less than $\lambda^{1/2}$ per derivative. In this case the symbol has regularity better than the natural scale of the oscillation}\label{goodreg}\end{figure}

We examine this case from the perspective of the proof of the Van der Corput lemma which is achieved in three steps:
\begin{enumerate}
\item First the region $|x|\leq{}\lambda^{-\alpha}$ $(\alpha\in\R^{+}$ is to be chosen later) is excised and that contribution to the integral estimated simply by considerations of size/measure (no oscillations are used). That is for some $\chi:\R^{+}\to \R$, smooth compactly supported and equal to one on $[-1/2,1/2]$ we write
$$I(\lambda)=I_{\alpha}^{1}(\lambda)+I_{\alpha}^{2}(\lambda)$$
where
\begin{align*}I_{\alpha}^{1}(\lambda)&=\int_{\R^{n}}e^{i\lambda x^{2}}a(\lambda,x)\chi(\lambda^{-\alpha}|x|)dx\\
&\text{and}\\
I_{\alpha}^{2}(\lambda)&=\int_{\R^{n}}e^{i\lambda x^{2}}a(\lambda,x)(1-\chi(\lambda^{-\alpha}|x|))dx.\end{align*}
Then we can estimate $I_{\alpha}^{1}(\lambda)$ by
\begin{align*}
|I_{\alpha}^{1}(\lambda)|&\leq \int_{\R^{n}}|a(\lambda,x)\chi(\lambda^{-\alpha}|x|)|dx\\
&\leq C \lambda^{\alpha n}.\end{align*}
This leaves $I_{\alpha}^{2}(\lambda)$ where the support of the integrand is cut away from the critical point.
\item For $|x|>\lambda^{-\alpha}$ the operator
$$L=\frac{1}{2i\lambda x}\frac{d}{dx}$$
has the property that $Le^{i\lambda\phi}=e^{i\lambda\phi}$. So integrating by parts
\begin{align*}
I_{\alpha}^{2}(\lambda)&=\int_{\R^{n}}L\left[e^{i\lambda x^{2}}\right]a(\lambda,x)(1-\chi(\lambda^{-\alpha}x))dx\\
&=-\int_{\R^{n}}e^{i\lambda x^{2}}L\left[a(\lambda,x)(1-\chi(\lambda^{-\alpha}x))\right]dx.\end{align*}
This integration by parts argument can by repeated as many times as necessary (say N). Each time we integrate by parts we release one power of $\lambda^{-1}$ however we lose powers of $\lambda$ due to derivatives hitting $a(\lambda,x)(1-\chi(\lambda^{-\alpha}|x|))$ and the factors of $\frac{1}{x}$.  Keeping track of that loss we obtain
$$|I_{\lambda}^{2}(\lambda)|\leq C_{N} \lambda^{-N} \lambda^{2\alpha N}\lambda^{\alpha n}.$$
\item The two contributions are compared and $\alpha$ chosen so that they are equal in magnitude. This occurs when $\alpha=1/2$.
\end{enumerate}
Note that any cut off function excising a $\lambda^{-1/2}$ region will have a regularity loss of $\lambda^{1/2}$ per derivative. When restricted to $|x|>\lambda^{-1/2}$ so  does the factor $1/x$. Therefore  the proof of the Van der Corput lemma is unchanged if the symbol has regularity loss up to $\lambda^{1/2}$. 

\begin{figure}[h!]
\includegraphics[scale=0.5]{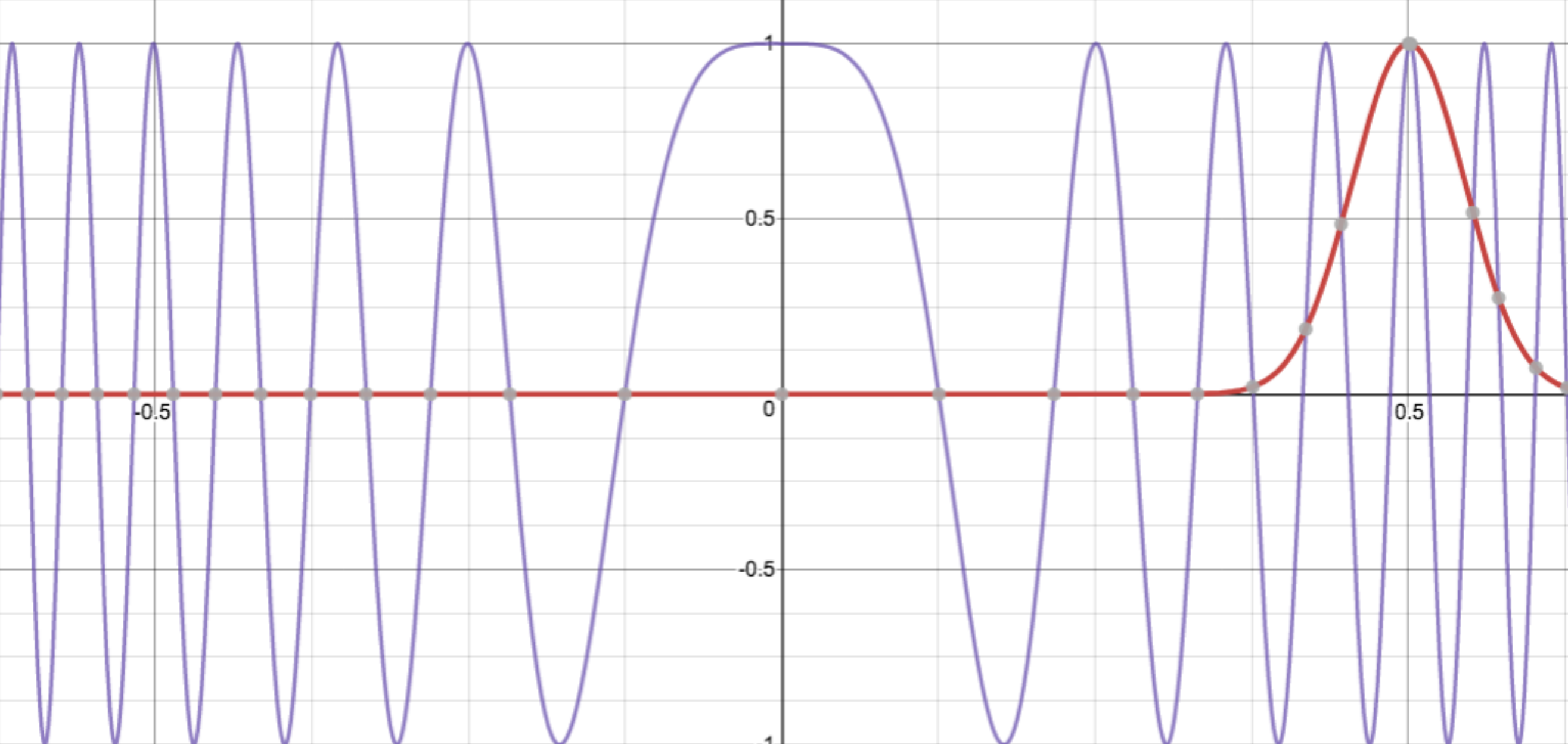}
\caption{The $y$ axes depicts both the real part of $e^{i\lambda x^{2}}$ and a symbol $a(\lambda,x)$ with regularity loss greater than $\lambda^{1/2}$ per derivative. In this case even though the symbol no longer has regularity better than the natural scale, we eventually see cancellation}\label{badreg}\end{figure}

In Figure \ref{badreg} we see an example where the regularity of $a(\lambda,x,y)$ is worse than the $\lambda^{1/2}$ scale. However, far enough from zero we are still able to gain some cancellation. Naturally we cannot expect to get any decay in the case
$$|D^{\gamma}a(\lambda,x,y)|\geq{}C_{\gamma}\lambda^{|\gamma|}$$
as then the oscillations are never rapid enough to induce cancellation. However we are able to obtain results if
$$|D^{\gamma}_{x}a(\lambda,x,y)|\leq{}C_{\gamma}\lambda^{|\gamma|\beta}$$
for $1/2\leq\beta\leq 1$.

We now consider the case $|\det(\Hess_{\phi})|\geq{}\mu^{n}$. First, to fix our ideas, we use the simplest model available to us, where $a(\lambda,x)$ no longer depends on $\lambda$ and so is denoted $a(x)$. Now
$$I(\lambda)= \int e^{i\lambda\mu x^{2}}a(x)dx.$$
If $a(x)$ has no regularity loss we could define a new parameter $\tilde{\lambda}=\mu\lambda$ and apply the Van de Corput lemma for $\tilde{\lambda}$. Therefore the fundamental length scale would become $\lambda^{-1/2}\mu^{-1/2}$. Now consider the cases where there is $\lambda^{\beta}$ loss per derivative (and assume $\beta>1/2$). Then we can obtain decay via integration by parts if
$$\left|\frac{1}{x}\frac{1}{\lambda\mu}\lambda^{\beta}\right|\leq{}1.$$
That is $|x|\geq{}\lambda^{-1+\beta}\mu^{-1}$. So we expect this to be the correct scale to replace $\lambda^{-1/2}$ in a Van de Corput style proof. For the general case we need to assume some control on higher order derivatives of $\phi$. This leads us to Definition \ref{def:ordermu} which locally characterises the type of critical point with which we can deal. Since the conditions are all local we work in a ball of radius $\epsilon$ about the critical point (denoted $B_{\epsilon}(x_{c})$) and allow all constants in inequalities to depend on $\epsilon$ (denoted $\geq_{\epsilon}$ and $\leq_{\epsilon}$). 

\begin{defn}\label{def:ordermu}
We say that on a set $X\subset\R^{n}$ a phase function $\phi$ has ``a single non-degenerate critical point to scale $\mu$'' if on $X$ there is a single point $x_{c}$ where $\nabla_{x}\phi(x_{c})=0$ and
\begin{itemize}
\item For any $\epsilon$ (small but not dependent on $\lambda$), 
$$|\nabla_{x}\phi(x)|\geq_{\epsilon}\mu\quad\text{for }x\in X\setminus B_{\epsilon}(x_c).$$
\item $$|\det(\Hess_{\phi})|_{x=x_{c}}|\geq{}C\mu^{n}.$$
\item For any multi-index $\gamma$,
$$|D_{x}^{\gamma}\phi|\leq C_{\gamma}\mu.$$

\end{itemize}
\end{defn} 

To further motivate Definition \eqref{def:ordermu} let's consider some examples and non-examples 
\begin{example}
The phase function
$$\phi(x,t,y)=t|x|^{2}-\langle y,x\rangle$$
for $x,y\in B_{1}(0)$ and $t\in [\lambda^{-1},1]$. This phase function frequently arises in the study of semiclassical Schr\"{o}dinger operators. It has a critical point
$$(x(t,y))_{j}=\frac{y_{j}}{2t}.$$
When $\left|x-\frac{y}{2t}\right|>\epsilon$, we have that 
$$|\nabla_{x}\phi|=|2tx -y|=2t\left|x-\frac{y}{2t}\right|>2t\epsilon$$
and $\Hess_{\phi}=(2t)^{n}$. Further
$$|D^{\gamma}_{x}\phi|\leq C_{\gamma}t.$$
Indeed once we have gone beyond second order derivatives $D^{\gamma}_{x}\phi=0$. Therefore this phase function satifies the conditions of Definition \ref{def:ordermu}

\end{example}

\begin{example}
The previous example did not have any explicity dependence on $\lambda$. We can weakly peturb it to allow some $\lambda$ dependence for example
$$\phi(\lambda,x,t,y)=t(1+q(\lambda^{-1}))|x|^{2}-\langle y,x\rangle$$
where $q$ is a polynomial. The critical point becomes
$$(x(t,y))_{j}=\frac{y_{j}}{2t(1+q(\lambda^{-1}))}.$$
When $|x-x(t,y)|>\epsilon$ we have
$$|\nabla_{x}\phi|>2t(1+q(\lambda^{-1}))\epsilon>C\epsilon t$$
where $C$ depends on the coefficients of $q$. Note further that the effect of adding $q(\lambda^{-1})$ doesn't trouble the derivative bounds at all. Therefore this functions satifies the conditions of Definition \ref{def:ordermu}.
\end{example}

\begin{nexample}
We cannot however add too great a peturbation to $\phi$. For instance
$$\phi(\lambda,x,t,y)=t|x|^{2}-\langle y,x\rangle +\lambda^{-2}\psi(\lambda x)$$
for $\psi$ smooth satisfies the lower bounds on the Hessian (when $t>\lambda^{-1}$) but does not satisfy the requirements of Defintion \ref{def:ordermu} as the higher order derivatives of $\phi$ can get very large. 
\end{nexample}

\begin{nexample}
If we write $x=(x_{1},x')$ the phase function
$$\phi(\lambda,x,t,y)=t(x_{1}^{3}+|x'|^{2})-\langle y,x\rangle$$
also fails Defintion \ref{def:ordermu} because when $x_{1}=0$ there is no lower bound on $\Hess_{\phi}$. 
\end{nexample}

With the assumptions of Definition \ref{def:ordermu} were are able to obtain an asymptotic expression for $I(\lambda,y)$ as given by Theorem \ref{thm:main}. 

\begin{thm}\label{thm:main}
Let
$$I(\lambda,y)=\int_{\R^{n}} e^{i\lambda\phi(\lambda,x,y)}a(\lambda,x,y)dx$$
where for each $y$ $\phi$ has a single non-degenerate critical point, denoted $x(y)$, to scale $\mu\geq\lambda^{-1+\beta}$ on the support (in $x$) of $a(\lambda,x,y)$. Further assume the constants from Definition \ref{def:ordermu} are uniform in $y$. Suppose the symbol $a(\lambda,x,y)$ obeys the regularity bounds
$$|D^{\gamma}_{x}a(\lambda,x,y)|\leq{}C_{\gamma}\lambda^{\beta|\gamma|}\quad\frac{1}{2}\leq\beta<1, \gamma\text{ any multi-index}.$$
Then
$$I(\lambda,y)=e^{i\lambda\phi(\lambda,x(y),y)}b(\lambda,y)$$
where
$$|D^{\gamma}_{y}b(\lambda,y)|\leq{}\mu^{-n}\lambda^{-n(1-\beta)}\lambda^{\beta|\gamma|}\sup_{|\sigma|\leq{}|\gamma|}\left|D^{\sigma}_{y}x(y)\right|^{|\sigma|},\quad\gamma\text{ any multi-index} .$$
\end{thm}

\section{Proof of Theorem \ref{thm:main}}

To prove Theorem \ref{thm:main} we first note that if $\tilde{\phi}(\lambda,x,y)=\phi(\lambda,x,y)-\phi(\lambda,x(y),y)$ then by expanding $\tilde{\phi}(\lambda,x,y)$ about $x(y)$ we have that
\begin{equation}\tilde{\phi}(\lambda,x,y)=[x-x(y)]^{T}\left[\Hess_{\phi}\big|_{x(y)}\right][x-x(y)]+O(\mu|x-x(y)|^{3}).\label{localform}\end{equation}
Note that we could use the Morse Lemma to obtain the nicer canonical form
$$\phi(\lambda,x,y)-\phi(\lambda,x(y),y)=\sum_{i=1}^{k}(x_{i}-x_{i}(y))^{2}-\sum_{i=k+1}^{n}(x_{i}-x_{i}(y))^{2}.$$
However this would only be valid for some $\mu$ dependent (but unspecified) region of $x(y)$. Since we want to allow the parameter dependence and we want $\mu$ to be able to depend on $\lambda$ we prefer to use \eqref{localform}. We now prove decay bounds for integrals
$$I(\lambda,y)=\int_{\R^{n}} e^{i\lambda\tilde{\phi}(\lambda,x,y)}|x-x(y)|^{2\gamma}d(\lambda,x,y)dx\quad\gamma\geq{}1$$
where $d(\lambda,x,y)$ is supported away from the critical point at $x(y)$ but otherwise has the same properties as $a(\lambda,x,y)$.

\begin{prop}\label{prop:decay}
Suppose 
$$I(\lambda,y)=\int_{\R^{n}} e^{i\lambda\tilde{\phi}(\lambda,x,y)}|x-x(y)|^{2\gamma}\chi(\mu\lambda^{\alpha}|x-x(y)|)d(\lambda,x,y)dx\quad\gamma\geq{}1,\text{ and }\alpha\leq{}1-\beta$$
where $\phi(\lambda,x,y)$ and $d(\lambda,x,y)$ have the same properties as the phase function/symbol in Theorem \ref{thm:main} and $\chi(r)$ is smooth with
$$\chi(r)=\begin{cases}
1& |r|\leq{}1\\
0& |r|\leq{}2.\end{cases}$$ 
Then
$$|I(\lambda,y)|\leq C_{N} (\mu^{-1}\lambda^{-\alpha})^{(n+2\gamma)}\lambda^{-N(1-\beta-\alpha)}$$
for any natural number $N$.
\end{prop}

\begin{proof}
Using the local representation \eqref{localform} for $\tilde{\phi}(x,y)$ we have that
$$\nabla_{x}\tilde{\phi}(\lambda,x,y)=\left[\Hess_{\phi}\big|_{x(y)}\right][x-x(y)]+[x-x(y)]^{T}\left[\Hess_{\phi}\big|_{x(y)}\right]+O(\mu|x-x(y)|^{3}).$$
Therefore by a partition of unity we can assume that there is some $k$ so that
$$\left|\frac{\partial\tilde{\phi}}{\partial x_{k}}\right|\geq{}c\mu|x-x(y)|$$
we may as well assume that $k=1$. Therefore we define the operator
$$L=\frac{1}{i\lambda\partial_{x_{1}}\tilde{\phi}}\frac{\partial}{\partial x_{1}}$$
so that
$$Le^{i\lambda\tilde{\phi}(\lambda,x,y)}=e^{i\lambda\tilde{\phi}(\lambda,x,y)}.$$
We use this to integrate by parts. That is since
$$I(\lambda,y)=\int \left(Le^{i\lambda\tilde{\phi}(\lambda,x,y)}\right)|x-x(y)|^{2\gamma}\chi(\mu\lambda^{\alpha}|x-x(y)|)d(\lambda,x,y)dx,$$
integration by parts yields
$$I(\lambda,y)=-\int e^{i\lambda\tilde{\phi}(\lambda,x,y)}L\Big(|x-x(y)|^{2\gamma}\chi(\mu\lambda^{\alpha}|x-x(y)|)d(\lambda,x,y)\Big)dx.$$
By repeating the argument
$$I(\lambda,y)=(-1)^{N}\int e^{i\lambda\tilde{\phi}(\lambda,x,y)}L^{N}\Big(|x-x(y)|^{2\gamma}\chi(\mu\lambda^{\alpha}|x-x(y)|)d(\lambda,x,y)\Big)dx.$$
Therefore to obtain the decay estimates all we need to do is examine the effect of $N$ operations of $L$ on the symbol. Now
$$L^{N}\Big(|x-x(y)|^{2\gamma}\chi(\mu\lambda^{\alpha}|x-x(y)|)d(\lambda,x,y)\Big)$$
is made up of terms of the form
$$\left(\frac{|x-x(y)|^{2\gamma-2N_{1}}}{\mu^{N_{1}}\lambda^{N_{1}}}\right)\left(\frac{|x-x(y)|^{-N_{2}+2\gamma}\lambda^{\beta N_{2}}}{\mu^{N_{2}}\lambda^{N_{2}}}\right)\left(\frac{|x-x(y)|^{-N_{3}+2\gamma}(\mu^{-1}\lambda^{\alpha})^{N_{3}}}{\mu^{N_{3}}\lambda^{N_{3}}}\right)$$
with $N_1+N_{2}+N_{3}=N$. With the restrictions that $\alpha\leq{}1-\beta$, and  $\mu\geq{}\lambda^{-\beta+1}$ the largest contribution comes from the terms of form
$$\frac{|x-x(y)|^{-N+2\gamma}\lambda^{\beta N}}{\mu^{N}\lambda^{N}}.$$
So
\begin{align*}
|I(\lambda,y)|&\leq{}\int_{|x-x(y)|\geq{}\mu^{-1}\lambda^{-\alpha}}\frac{|x-x(y)|^{-N+2\gamma}\lambda^{\beta N}}{\mu^{N}\lambda^{N}}dx\\
&\leq{}\mu^{-1-2\gamma}\lambda^{-N(1-\beta-\alpha)-\alpha(2\gamma+n)}\\
&=(\mu^{-1}\lambda^{-\alpha})^{(n+2\gamma)}\lambda^{-N(1-\beta-\alpha)}.\end{align*}

\end{proof}

We are now in a position to prove Theorem \ref{thm:main}

\begin{proof}[Proof of Theorem \ref{thm:main}]
We write $I(\lambda,y)$ as
$$I(\lambda,y)=e^{i\lambda\phi(\lambda,x(y),y)}b(\lambda,y)$$
where
$$b(\lambda,y)=\int_{\R^{n}} e^{i\lambda\tilde{\phi}(\lambda,x,y)}a(\lambda,x,y)dx$$
and
$$\tilde{\phi}(\lambda,x,y)=\phi(\lambda,x,y)-\phi(\lambda,x(y),y).$$
Recall that we have the local form \eqref{localform} for $\tilde{\phi}(\lambda,x,y)$ of
$$\tilde{\phi}(\lambda,x,y)=[x-x(y)]^{T}\left[\Hess_{\phi}\big|_{x(y)}\right][x-x(y)]+O(\mu|x-x(y)|^{3}).$$
Note that
\begin{multline}
\frac{\partial\tilde{\phi}}{\partial y_{k}}(\lambda,x,y)=-\Big([\partial_{y_{k}}x(y)]^{T}\left[\Hess_{\phi}\big|_{x(y)}\right][x-x(y)]\\
+[x-x(y)]^{T}\left[\Hess_{\phi}\big|_{x(y)}\right][\partial_{y_{k}}x(y)]\Big)+O(\mu|x-x(y)|^{2}).\label{yderiv}\end{multline}
On the other hand
\begin{equation}\frac{\partial\tilde{\phi}}{\partial x_{k}}=e_{k}^{T}\left[\Hess_{\phi}\big|_{x(y)}\right][x-x(y)]+[x-x(y)]^{T}\left[\Hess_{\phi(y)}\right]e_{k}+O(\mu|x-x(y)|^{2})\label{xderiv}\end{equation}
where $e_{k}$ is the standard basis vector with $1$ in the $k$-th position. 
So
\begin{equation}\frac{\partial\tilde{\phi}}{\partial y_{k}}(\lambda,x,y)=[\partial_{y_{k}}x(y)]^{T}[\partial_{x}\tilde{\phi}]+O(\mu|x-x(y)|^{2}).\label{xyconnect}\end{equation}
Now
$$\frac{\partial b(\lambda,y)}{\partial y_{k}}=\int_{\R^{n}} e^{i\lambda \tilde{\phi}(\lambda,x,y)}\left(i\lambda\frac{\partial \tilde{\phi}(x,y)}{\partial y_{k}}a(\lambda,x,y)+\frac{\partial a(\lambda,x,y)}{\partial y_{k}}\right)dx$$
and so using \eqref{xyconnect} we have that
\begin{align*}
\frac{\partial b(\lambda,y)}{\partial y_{k}}&=\int_{\R^{n}} e^{i\lambda\tilde{\phi}(x,y)}\left(i\lambda[\partial_{y_{k}}x(y)]^{T}[\partial_{x}\tilde{\phi}]a(\lambda,x,y)+\frac{\partial a(\lambda,x,y)}{\partial y_{k}}+\lambda|x-x(y)|^{2}G_{\mu}(x,y)\right)dx\\
&=\int_{\R^{n}} e^{i\lambda\tilde{\phi}(\lambda,x,y)}\left([\partial_{y_{k}}x(y)]^{T}[\partial_{x}a(\lambda,x,y)]+\frac{\partial a(\lambda,x,y)}{\partial y_{k}}+\lambda\mu|x-x(y)|^{2}G(\lambda,x,y)\right)dx\end{align*}
where to get the second line we have integrated by parts in $x$. Therefore the regularity properties of $a(\lambda,x,y)$ tell us that
$$\frac{\partial b(\lambda,y)}{\partial y_{k}}=\int_{\R^{n}} e^{i\lambda\tilde{\phi}(\lambda,x,y)}\left((\lambda \mu)^{\beta}d(\lambda,x,y)+\lambda\mu|x-x(y)|^{2}G(\lambda,x,y)\right)dx$$
where $d(\lambda,x,y)$ has the same regularity properties as $a(\lambda,x,y)$. We can repeat the process to get extra derivatives, each at a cost of one factor of $\lambda^{\beta}$ or $\lambda\mu|x-x(y)|^{2}$. Therefore to obtain an estimate on $|D^{\gamma}_{y}b|$ we need to control terms of the form
$$\lambda^{\gamma_{1}}\mu^{\gamma_{1}}\left|\int_{\R^{n}} e^{i\lambda\tilde{\phi}(\lambda,x,y)}|x-x(y)|^{2\gamma_{1}}\lambda^{\beta \gamma_{2}}d(\lambda,x,y)dx\right|$$
where $\gamma_{1}+\gamma_{2}=\gamma$. We will argue similar to the Van der Corput Lemma and excise a region $|x-x(y)|\leq{}\mu^{-1}\lambda^{-1+\beta}$. From Proposition \ref{prop:decay} we know that
$$\left|\int_{\R^{n}} e^{i\lambda\tilde{\phi}(\lambda,x,y)}|x-x(y)|^{2\gamma_{1}}\lambda^{\beta \gamma_{2}}d(\lambda,x,y)\chi(\mu\lambda^{1-\beta}|x-x(y)|)dx\right|\leq{}(\mu^{-1}\lambda^{-1+\beta})^{(n+2\gamma_{1})}\lambda^{\beta\gamma_{2}}.$$
From the support properties of $\chi$ we also have
$$\left|\int_{\R^{n}} e^{i\lambda\tilde{\phi}(\lambda,x,y)}|x-x(y)|^{2\gamma_{1}}\lambda^{\beta \gamma_{2}}d(\lambda,x,y)(1-\chi(\mu\lambda^{1-\beta}|x-x(y)|))dx\right|\leq{}(\mu^{-1}\lambda^{-1+\beta})^{n+2\gamma_{1}}\lambda^{\beta\gamma_{2}}.$$
So
\begin{align*}
|D^{\gamma}_{y}b|&\leq{}\lambda^{\gamma_{1}}\mu^{\gamma_{1}}\lambda^{\beta\gamma_{2}}(\mu^{-1}\lambda^{-1+\beta})^{n+2\gamma_{1}}\\
&\leq \mu^{-n}\lambda^{-n(1-\beta)}\lambda^{\beta(\gamma_{2}-\gamma_{1})}\lambda^{-2\beta\gamma_{1}}\end{align*}
where in the last line we have use the restriction that $\mu\geq{}\lambda^{-1+\beta}$. Therefore the largest contribution is when $\gamma=\gamma_{2}$ and
$$|D_{y}^{\gamma}b|\leq\mu^{-n}\lambda^{-n(1-\beta)}\lambda^{\beta\gamma}$$
as required. 

\section*{Appendix - technical terminology}

\begin{itemize}
\item[$\gamma$:] A multi-index $\gamma=(\gamma_{1},\dots,\gamma_{n})$, $\gamma_{j}\in\N$.
\item[$|\gamma|$:] The size of the multi-index $|\gamma|=\gamma_{1}+\cdots+\gamma_{n}$. 
\item[$D^{\gamma}_{x}$:] The differential operator
$$D^{\gamma}_{x}=\left(\frac{1}{i}\frac{\partial}{\partial x_{1}}\right)^{\gamma_{1}}\cdots\left(\frac{1}{i}\frac{\partial }{\partial x_{n}}\right)^{\gamma_{n}}$$
\item[$\beta$:] A real number $\frac{1}{2}\leq \beta<1$ that measures the loss of regularity per derivative.
\item[$B_{r}(p)$:] The Euclidean ball of radius $r$ about point $p$
\item[$\geq_{\epsilon},\leq_{\epsilon}$:] Inequalities which allow the consants to depend on $\epsilon$. That is
$$|f(x)|\leq_{\epsilon}\lambda\quad\text{ means } |f(x)|\leq C_{\epsilon}\lambda$$
where the constant $C_{\epsilon}$ may depend on $\epsilon$ (but not $\lambda$).
\item[$\approx$:] Approximately equal to. In particular has the same, up to constants, asymptotic growth rate (in $\lambda$)
\item[$\mu$:] A parameter $\mu>0$ that may depend on $\lambda$ and $y$ which controls the determinant of the Hessian from below in the following fashion
$$|\det(\Hess_{\phi})|>c\mu^{n}.$$
The parameter $\mu$ is allowed to approach zero as $\lambda\to\infty$. 
\item[$x_{c}$:] Notation for the critical point
$$\nabla_{x}\phi(x_{c})=0$$
or
$$\nabla_{x}\phi(\lambda,x_{c})=0$$
in the case where $\phi$ has no $y$ dependence. 
\item[$x(y)$:] Notation for the critical point
$$\nabla_{x}\phi(\lambda,x(y),y)=0$$
for the case where $\phi$ depends on $y$. 

\end{itemize}

\bibliography{references}
\bibliographystyle{plain}

\end{proof}
\end{document}